\documentclass{amsart}[12pt]

\usepackage{hyperref}
\usepackage{etoolbox}
\usepackage{mathrsfs}
\usepackage{graphicx}
\usepackage{color,latexsym,amsfonts,amsthm,amsmath,amssymb,path,enumitem}
\AtBeginEnvironment{quote}{\smaller}

\newtheorem{thm}{Theorem}[section]
\newtheorem{lem}[thm]{Lemma}

\newtheorem{rem}[thm]{Remark}

\usepackage{dsfont}
\newcommand{\Z}{\mathbb{Z}}

\def\nin{\noindent}

\def\.{\hskip.06cm}
\def\ts{\hskip.03cm}

\def\.{\hskip.06cm}
\def\ts{\hskip.03cm}
\def\mts{\hspace{-.04cm}}

\def\zz{\mathbb Z}

\def\tT{\mathbb T}

\def\Ga{R}

\def\<{\langle}
\def\>{\rangle}

\def\pt{\partial}

\def\SP{{{\rm{\textsf{\#P}}}}}

\def\NP{{{\rm{\textsf{NP}}}}}

\begin{document}
\title[Fast domino tileability]{Fast domino tileability}

\author[Igor~Pak]{ \ Igor~Pak$^\star$}
\author[Adam~Sheffer]{ \ Adam~Sheffer$^\dagger$}
\author[Martin~Tassy ]{ \ Martin~Tassy$^\star$}

\begin{abstract}
\emph{\mts Domino tileability} \ts is a classical problem in Discrete Geometry,
famously solved by Thurston for simply connected regions in nearly
linear time in the area.  In this paper, we improve upon Thurston's
height function approach to a nearly linear time in the perimeter.
\end{abstract}

\thanks{\thinspace ${\hspace{-.45ex}}^\star$Department of Mathematics,
UCLA, Los Angeles, CA, 90095.
\hskip.06cm
Email:
\hskip.06cm
\texttt{\{pak,\ts{}mtassy\}@math.ucla.edu}}

\thanks{\thinspace ${\hspace{-.001ex}}^\dagger$Department of Mathematics,
Caltech, Pasadena, CA, 91125. \hskip.06cm
Email:
\hskip.06cm
\texttt{adamsh@gmail.com}.}

\maketitle

\vskip.9cm

\section{Introduction}\label{s:intro}

\nin
Given a region $\Ga$ and a set of tiles~$T$, decide whether~$\Ga$
is tileable with copies of the tiles in~$T$.  This is a classical
\emph{tileability problem}, occupying a central stage in Discrete
and Computational Geometry.  For general sets of tiles, this is
a foundational problem in Computability~\cite{Ber,Boas} and Computational
Complexity~\cite{GJ,Lev,V1}.  For domino tiles, the problem is a special
case of the \emph{Perfect Matching} problem.  It can be solved in
polynomial time; even the counting problem can be solved at the
cost of matrix multiplication (see e.g.~\cite{LP,Ken}).

In~1990, Thurston pioneered a new approach to the subject based
on the study of \emph{height functions}, which can be viewed as
integer maps on the regions~\cite{Thu}.  Thurston outlined a
domino tileability algorithm which later has been carefully
analyzed (see~$\S$\ref{ss:fin-cs}) and significantly extended
to many other tileability problems (see~$\S$\ref{ss:fin-tile}).

\begin{thm}[Thurston, 1990] \label{t:thurston}
Let $\Ga$ be a simply connected region in the plane~$\zz^2$, and
let $n = |\Ga|$ be the area of $\Ga$.
There exists an algorithm that decides tileability of~$\Ga$
in time $O(n\log n)$.
\end{thm}

This result is worth comparing to the classical Hopcroft--Karp
algorithm which has $O(n^{3/2})$ time for testing whether
a bipartite graph with $n$ vertices and bounded degree
has a perfect matching.  This general bound has been
significantly improved in recent years (see~$\S$\ref{ss:fin-match}).

Note that for polynomial time problems the cost of the algorithm
depends heavily on how the input is presented.   In case of
graphs, the input is a list of vertices and edges, of size~$\Theta(n)$.
On the other hand, the region $\Ga$ is traditionally presented
as a list of squares, thus of size $\Theta(n \log n)$.

\smallskip

The main idea behind this work is that plane regions $\Ga$ can
be presented by the set of boundary squares.  The input then has
size $\Theta(p \log p)$, where $p = |\pt \Ga|$ is the perimeter of~$\Ga$.
More economically, for simply connected~$\Ga$ the input can be presented
as a sequence of directed boundary edges \textsc{Right}, \textsc{Up},
\textsc{Left} and \textsc{Down}, starting at the origin.  Of course,
in this case the boundary set of squares can be computed in time
$\Theta(p \log p)$.  Either way, one can ask if in this presentation
one can improve upon Thurston's algorithm.  Here is our main result:

\begin{thm} \label{t:main}
Let $\Ga$ be a simply connected region in the plane~$\zz^2$, and
let $p = |\pt \Ga|$ be the perimeter of $\Ga$.
There exists an algorithm that decides tileability of~$\Ga$
in time $O(p \ts\log\mts p)$.
\end{thm}

The result gives the first improvement over Thurston's algorithm
in~25 years, and is nearly optimal in this presentation.
Clearly, the perimeter $p= \Omega(\sqrt{n})$ can be significantly
smaller than~$n$, so Theorem~\ref{t:main} improves upon Thurston's
for all regions with $p = o(n)$.

\smallskip

Let us note that Thurston's algorithm not only decides domino
tileability, but also constructs a domino tiling when the region
is tileable.  Specifically, Thurston shows that for every tileable
region~$\Ga$ there is a unique \emph{maximum tiling} $T_{\circ}$,
corresponding to the \emph{maximum height function} of the region~$\Ga$.
He then inductively computes $T_\circ$ in time $O(n \log n)$.
Clearly, it would be unhelpful to match this result, since listing
all dominoes requires $\Omega(n\log n)$ time.  However, we can do this in
the following oracle model.

A tiling $T$ of a region $\Ga$ is said to be \emph{site-computable}
with a \emph{query cost}~$t$ if after preprocessing, for every square
$x \in \Ga$, we can compute the adjacent square~$y$ of a domino in~$T$
in time~$t$.  See~$\S$\ref{ss:fin-oracle} for the reasoning behind
this model.

\begin{thm} \label{t:main-inter}
Let $\Ga$ be a simply connected region in the plane~$\zz^2$, and
let $p = |\pt \Ga|$ be the perimeter of $\Ga$.  If $\Ga$ is tileable
with dominoes, the maximum tiling $T_{\circ}$ is site-computable
with a preprocessing time $O(p \log p)$, and a
query time $O(\log^2 p)$.
\end{thm}

\smallskip

The rest of the paper is structured as follows.  In the next lengthy
Section~\ref{s:tileability} we present a criterion for domino
tileability in terms of the height function on the boundary.
The algorithm is presented in the following Section~\ref{s:algorithm}.
We then show that similar results also hold for a triangular lattice
(Section~\ref{s:other}).  We conclude with final remarks and open
problems in Section~\ref{s:fin}.

\bigskip

\section{Tileability theorems} \label{s:tileability}

\nin
In this section we present necessary and sufficient conditions for the domino tileability of simply connected regions in $\Z^{2}$.
In the following section, we use these conditions to construct an algorithm for checking the tileability of a simply connected region.
For the rest of this section we assume that $\Z^{2}$ is endowed with a chessboard coloring, such that the square with corners $(0,0)$ and $(1,1)$ is white.

Let $R$ be a simply connected region of $\Z^{2}$ such that the origin is on the boundary of $R$ and  $T_{R}$ is a tiling of $R$.
A classic result (for example, see Fournier \cite{Fou}) states that there exists a height function $h:R \rightarrow \Z$ that corresponds to $T_{R}$, and is defined as follows:
\begin{itemize}
\item $h(0,0)=0$, and
\item for an edge $(x,y)$ in $T_{R}$, such that when crossing from $x$ to $y$ there is a white square to our left, we have $h(x)-h(y)=1$.
\end{itemize}

We denote by $\partial R$ the points of $\Z^2$ that are on the boundary of $R$. To obtain a height function $h: \partial R \rightarrow \Z$, we start from $(0,0)$ and travel along the boundary. We begin by setting $h(0,0)=0$ and then place a height value to each point that we visit, according to the second condition of the definition. That is, when we cross an edge from point $x$ to point $y$, if there is a white square to our left we set $h(y)=h(x)-1$, and otherwise $h(y)=h(x)+1$. If $R$ is not tileable, the final edge in our trip might not satisfy the condition. When this happens we say that $\partial R$ has no valid height function.

The following well known lemma characterizes the functions $h:R\to \Z$ that are height functions of some tiling of~$R$.

\begin{lem}[see~\cite{Fou}] \label{le:CorrespondinghFunc}
Let $R$ be a simply connected region of $\Z^{2}$ with the origin on its boundary, and let $\mathscr{H}_{R}$ be the set of height functions that correspond to tilings of $R$. Then a function $h:R \rightarrow \Z$ is in $\mathscr{H}_{R}$ if and only if the following conditions hold.
\begin{enumerate}[label={(\roman*)}]
\item For every $x,y \in R$ such that $x=(a_{1},b_{1})$, $y=(a_{2},b_{2})$, and $a_{1}-a_{2} =b_{1}-b_{2} =0 \mod 2$, we have $h(x)-h(y) = 0 \mod 4$.
\item For every edge $(x ,y) \in \partial R$ such that when crossing from $x$ to $y$ there is a white square to our left, we have $h(x)-h(y)= 1$.
\item For every edge $(x,y)\in R$, we have $|h(x)-h(y)|\leq 3$.
\end{enumerate}
\end{lem}

One can also consider a tiling of the entire plane $\Z^{2}$. Lemma~\ref{le:CorrespondinghFunc} remains valid for such tilings, with Condition~(ii) becoming redundant. In this case, the set of height functions $\mathscr{H}_{\Z^{2}}$ that are zero at a given point $x=(a,b)$ has a maximum element $\alpha (x,\cdot)$ (i.e., for every $p\in \Z^2$, no height function $h \in \mathscr{H}_{\Z^{2}}$ satisfies $h(p)> \alpha(x,p)$) that is defined as follows. For a point $y\in \Z^{2}$, we write $y=x+(i,j)$ and set $\delta (i,j)=  i-j \mod 2$. If $a-b=0 \mod 2$, then
\begin{equation*}
  \alpha (x,y)= \left\{
    \begin{split}
     2\| y-x \|_{\infty}+ \delta (i,j),  \qquad \text{ if } i \geq j, \\
     2\| y-x \|_{\infty}- \delta (i,j),  \qquad \text{ if } i < j.    \end{split}
  \right.
\end{equation*}
If $a-b=1 \mod 2$, then
\begin{equation*}
  \alpha (x,y)=\left\{
    \begin{split}
     2\| y-x \|_{\infty}- \delta (i,j), \qquad  \text{ if } i \geq j, \\
    2\| y-x \|_{\infty}+ \delta (i,j), \qquad  \text{ if } i < j.    \end{split}
  \right.
\end{equation*}
It is easy to see that $\alpha (x,\cdot)$ satisfies conditions~(i) and~(iii) of Lemma~\ref{le:CorrespondinghFunc}, and is thus a height function that corresponds to a tiling of $\Z^2$.
Examples of maximal tilings that correspond to $\alpha (x,\cdot)$ are depicted in Figure~\ref{fi:Maximal} (left and center).

\begin{figure}[h]
\centering
\includegraphics[scale=0.57]{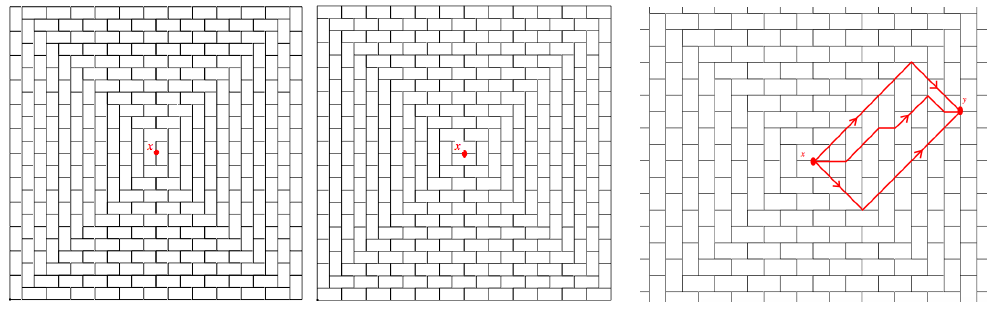}
\caption{Left and center: Maximal tilings that are centered at $x$ and correspond to the height function $ \alpha (x,\cdot)$. Right: Geodesic paths between $x$ and~$y$.}
\label{fi:Maximal}
\end{figure}

We say that a sequence of points $(x_{1},\ldots, x_{n})$ is a \emph{geodesic path} if
\begin{itemize}
\item  For every $i < n$, we have $\| x_{i+1} - x_1 \|_{\infty} = \| x_{i} - x_1 \|_{\infty}$+1.
\item  For every $i < n$, the points  $x_{i}$ and $x_{i+1}$ are corners of a common $1\times 1$ square in $\Z^{2}$.
 \end{itemize}
The right part of Figure \ref{fi:Maximal} depicts several geodesic paths between the same pair of points.

One key observation in our proofs is that $\alpha (x,\cdot)$ is strictly increasing on any geodesic path $(x_{1},\ldots,x_{n})$.
By the conditions of Lemma \ref{le:CorrespondinghFunc}, for any $1 \le i < n$ and height function $h$ that is defined on $x_i$ and $x_{i+1}$, there are exactly two possible values for $h(x_{i+1})-h(x_i)$.
One of these two values is negative, and the other is $\alpha (x_{1},x_{i+1}) -\alpha (x_{1},x_{i})$. That is, for any height function $h$ we have $h(x_{i+1})-h(x_i) \le \alpha (x_{1},x_{i+1}) -\alpha (x_{1},x_{i})$.

Another useful observation is that for every geodesic path $(x_{1},\ldots,x_{n})$ and $1 < i < n$, we have that $(x_{1},\ldots, x_{i})$ and $(x_{i},\ldots, x_{n})$ are also geodesic paths. By combining this with the previous observation, we notice that $\alpha(x,\cdot)$ is additive on geodesic paths. That is, if $(x_{1},\ldots,x_{n})$ is a geodesic path and $1 < i < n$, then $\alpha (x_{1},x_{i}) +\alpha (x_{i},x_{n})= \alpha (x_{1},x_{n})$.

 For a simply connected region $R$ and two points $x,y \in R$, we write $x \sim_{R} y$ if there exists a geodesic path between $x$ and~$y$ that is fully in~$R$.
The following lemma gives a necessary and sufficient condition for the tileability of a simply connected region~$R$.
This condition depends only on the height differences between pairs of points on the boundary $\partial R$.

\begin{lem} \label{l:necessary}
Let $R$ be a simply connected region of $\Z^{2}$ that contains the origin and let $h:\partial R \to \Z$ be a valid height function of $\partial R$.  The region $R$ is tileable if and only if for every pair $x,y \in \partial R$ that satisfies $x \sim_{R} y$, we have
\[
h(y) - h(x) \leq \alpha (x,y).
\]
\end{lem}

\begin{proof}
We first prove that the condition is necessary. Assume that $R$ is tileable and extend the domain of $h$ to all of $R$ according to a specific tiling of $R$. Let $(x_{1},...,x_{n})$ be a geodesic path between two points $x,y\in \partial R$ (that is, $x_1=x$ and $x_n=y$). Recall that for every $1\le i <n$ we have $h (x_{i+1}) -  h (x_{i}) \le \alpha (x,x_{i+1}) - \alpha (x,x_{i})$, which implies
\[
h(y)-h(x)  =  \sum_{i=1}^{n-1} \big(h (x_{i+1}) -  h (x_{i})\big)  \leq  \sum_{i=1}^{n-1} \big(\alpha (x,x_{i+1}) - \alpha (x,x_{i})\big)  = \alpha (x,y) - \alpha (x,x).
\]
Since $ \alpha (x,x)=0$, we get that $h(y)-h(x) \leq \alpha (x,y)$ which completes this part of the proof.

We next prove that the condition of the lemma is sufficient. For that, we show that the function
\begin{equation} \label{eq:hmax}
h_{\max}(y) \, = \, \min_{x \in \partial R \,,\, x\sim_{R} y} \. \bigl[h(x)+ \alpha(x,y)\bigr]
\end{equation}
satisfies the three conditions of Lemma~\ref{le:CorrespondinghFunc} (with respect to $R$).
For Condition~(ii), it suffices to show that for every $y\in \partial R$ we have $h_{\max}(y) = h(y)$ (since $h$ satisfies this condition by definition). Consider such a point $y\in \partial R$, then $h(y)+\alpha(y,y) = h(y)$  and the assumptions of the theorem implies that  for all $x \in \partial R$ such that $ x\sim_{R} y$ the inequality $h(x)+ \alpha(x,y) \geq h(y)$ holds. Hence $h_{\max}(y) = h(y)$ on $\partial R$, and Condition~(ii) is satisfied by $h_{\max}$.

For every $x\in \partial R$, the function $h+\alpha (x,\cdot)$ satisfies Condition~(i) since both $h$ and $\alpha (x,\cdot)$ are height functions. Since $h_{\max}(y) = h(y)$ for every $y\in \partial R$, the function $h_{\max}$ satisfies Condition~(i) on $\partial R$. This ``forces'' the various functions $\alpha (x,\cdot)$ to be identical $\mod 4$ on $\partial R$, and thus all over $R$. That is, for any $y\in \partial R$ the expression $h(x)+\alpha(x,y) \mod 4$ does not depend on the choice of $x\in \partial R$. This in turn implies that $h_{\max}$ satisfies Condition~(i).

It remains to prove that $h_{\max}$ satisfies Condition~(iii); that is, to show that for every pair of adjacent points $x,y\in R$, we have $|h_{\max}(x)-h_{\max}(y)| \leq 3$. If both $x$ and $y$ are in $\partial R$, this is immediate from Condition~(ii). Thus, without loss of generality, we assume that $x$ is in the interior of $R$.
Let $z \in \partial R$ satisfy $h_{\max}(x) = h(z)+\alpha(z,x)$ and let $P=(x_1,\ldots,x_n)$ be a geodesic path from $z$ to $x$ that is contained in $R$. Let $x_i$ be the last vertex in $P$ that is on $\partial R$, and notice that $P'=(x_i,\ldots,x_n)$ is a geodesic path from $x_i$ to $x$. Since $h_{\max}(x)$ has the maximum increase rate that any height function may have, we obtain that $h_{\max}(x) = h(x_i)+\alpha(x_i,x)$. Since $P'$ does not contain edges of $\partial R$, there must also exist a geodesic path from $x_i$ to $y$. Since $x$ and $y$ are neighbors, this implies $h_{\max}(y) \leq  h(x_i)+\alpha(x_i,y) \leq  h(x_i)+\alpha(x_i,x)+3$ and $h_{\max}(y)-h_{\max}(x) \le 3$. A symmetric argument yields $h_{\max}(x)-h_{\max}(y) \le 3$, and completes the proof of Condition~(iii).
\end{proof}

\begin{figure}[h]
\centering
\includegraphics[scale=0.68]{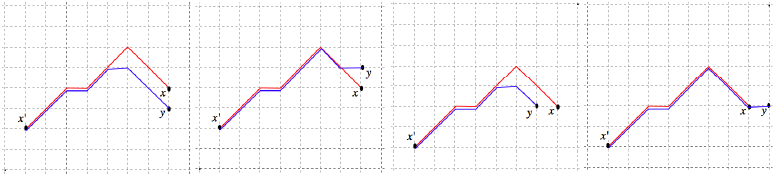}
\caption{If $y$ is adjacent to the interior point $x$, then a geodesic path from $x'$ to $x$ implies a geodesic path of a similar length from $x'$ to $y$. }
\label{fi:paths}
\end{figure}

For points $x,y\in \partial R$, we denote by $G(x,y)$ the set of points in $\Z^2$ that are in at least one geodesic path between $x$ and~$y$.
Notice that $G(x,y)$ is a rectangle with edges of slopes~{$\pm 1$}, possibly with two opposite
corners truncated; for example, see the right part of Figure \ref{fi:Maximal}.
Let $S$ be a set that contains $\partial R$ and any number of points from the interior of~$R$.
We write $x\approx_S y$ if $x,y \in S$ and $G(x,y)\setminus\{x,y\}$ is disjoint from~$S$.
The following theorem is a refinement of Lemma~\ref{l:necessary}, which reduces the number of point pairs that determine whether a region is tileable.

\begin{thm}
\label{tilability}
Let $R$ be a simply connected region of $\Z^{2}$ that contains the origin, let $h:\partial R \to \Z$ be a valid height function, and let $S \subset R $ be a subset that contains $\partial R$.  The region $R$ is tileable if and only if there exists a function $g:S \to \Z$ such that $g=h$ on $\partial R$ and for every pair $x,y \in S$ that satisfies $x\approx_{S} y$, we have
\begin{equation} \label{eq:condition}
-\alpha (y,x) \leq g(y) - g(x) \leq \alpha (x,y).
\end{equation}
\end{thm}
\begin{proof}
We first prove that the condition is necessary. Assume that $R$ has a tiling $T$ and let $g:R \to \Z$ be corresponding height function. By definition, $g=h$ on $\partial R$ and $g$ satisfies \eqref{eq:condition} for every $x,y\in \partial R$ with $x\sim_R y$. For every pair $x,y \in  \partial S$ for which $x\approx_{S} y$ there exists a geodesic path $(x_{1},\ldots, x_{n})$ in $R$ with $x=x_{1}$ and $y=x_{n}$. As in the proof of Lemma \ref{l:necessary}, we have
\[
g(y)-g(x)  =  \sum_{i=1}^{n-1} \big(g (x_{i+1}) -  g (x_{i})\big)  \leq  \sum_{i=1}^{n-1} \big(\alpha (x,x_{i+1}) - \alpha (x,x_{i})\big)  = \alpha (x,y).
\]
A symmetric argument implies $g(x)-g(y)\le \alpha (y,x)$, which completes the proof of this part.

To prove that the condition of the theorem is sufficient, we show that it implies the condition of Lemma~\ref{l:necessary}.
That is, if a function $g$ satisfies \eqref{eq:condition} for every pair $x,y \in  \partial S$ with $x\approx_{S} y$, then the same condition  is also satisfied for every pair $x,y \in  \partial R$ with $x \sim_{R} y$.

Consider a pair $x,y \in  S$ such that $x \sim_{R} y$.
We prove that \eqref{eq:condition} holds for $x,y$ by induction on $\|x-y \|_{\infty}$.
Since $\partial R \subset S$, this would complete the proof of the theorem.
For the induction basis, consider the case where $\|x-y \|_{\infty}=1$.
In this case we have $x\approx_S y$, so \eqref{eq:condition} is satisfied for $x,y$ by the definition of $g$.

For the induction step, consider the case where $\|x-y \|_{\infty}=k>1$.
In this case, either $x\approx_{S} y$  and \eqref{eq:condition} is satisfied by the definition of $g$, or there exists a geodesic path $P=(x_1,\ldots,x_n)$ between $x$ and~$y$ that is in $R$ and intersects $S\setminus\{x,y\}$. In the latter case, let $x_i$ be a vertex of $P$ that is in $S\setminus\{x,y\}$. Then $(x_1,\ldots,x_i)$ is a geodesic path between $x$ and $x_i$ and $(x_i,\ldots,x_n)$ is a geodesic path between $x_i$ and~$x_n$. By the induction hypothesis, we have
$$
\begin{aligned}
\alpha (x_{i},x) &\leq g(x_i) - g(x) \leq \alpha (x,x_i), \\
\alpha (y,x_{i}) &\leq g(y) - g(x_i) \leq \alpha (x_i,y).
\end{aligned}
$$
By combining these two inequalities we get $g(y) - g(x) \leq \alpha (x,x_i) +\alpha (x_i,y)$. Since $P$ is obtained by combining a geodesic path from $x$ to $x_i$ together with a geodesic path from $x_i$ to $y$, we have $\alpha (x,x_i) +\alpha (x_i,y) = \alpha(x,y)$, which in turn implies $g(y) - g(x) \le \alpha(x,y)$. A symmetric argument implies $g(x) - g(y)\le \alpha(y,x)$, which completes the induction step and the proof of the theorem.
\end{proof}

\section{Algorithm for tileability}\label{s:algorithm}

\subsection{Outline}
In this section we prove theorems~\ref{t:main} and~\ref{t:main-inter}.
First, we present an algorithm for checking whether a simply connected region $R$ is tileable.
The algorithm is based on partitioning $R$ into interior-disjoint squares of various sizes.
These squares have their vertices in $\Z^2$, but are ``rotated by $45^\circ$'' in the sense that the slopes of their edges are $\pm 1$.
To cover $R$ with such squares, along the boundary of $R$ we use right-angled triangles with two edges of length 1, instead of squares; for example, see Figure \ref{fi:Subdivision}.
We consider the set $S$ that consists of $\partial R$ together with the vertices of the rotated squares.
By Theorem \ref{tilability}, to check whether $R$ is tileable it suffices to compare between pairs $x,y \in S$ that satisfy $x\approx_{S} y$.
We will prove that each point of $S$ participates in at most eight such pairs, which would in turn imply that total number of pairs that satisfy $x\approx_{S} y$ is at most linear in the perimeter of~$R$.

\begin{figure}[h]
\centering
\includegraphics[scale=0.52]{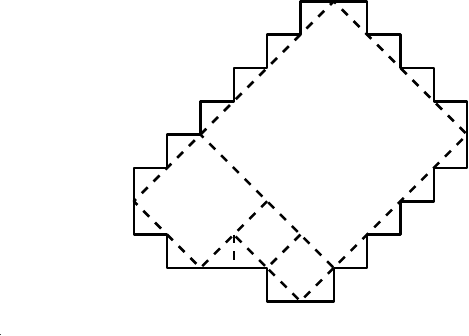}
\caption{A subdivision of the area that is bounded by the solid edges into rotated squares and right-angled triangles with two edges of length 1.}
\label{fi:Subdivision}
\end{figure}

\subsection{Partitioning the region}
We begin with the following technical result.

\begin{thm}
\label{th:Subdivision}
Let $R$ be a simply connected region with $|\partial R|=p$. Then there exists a subdivision $S$ of $R$ into $O(p)$ interior-disjoint rotated squares and right-angled triangles with two edges of length 1. Such a subdivision can be found in $O(p\log p)$ time.
\end{thm}

\begin{proof}
We prove the theorem by presenting an algorithm that receives a simply connected region $R$ with perimeter $p$ and constructs a subdivision $S$ of $R$ in time $O(p \log p)$. The subdivision $S$ consists of $O(p)$ rotated squares and right-angled triangles with two edges of length 1. All of the squares in this proof are rotated by $45^\circ$.

We  begin the algorithm by initializing several variables. Set $Q_\text{in}=Q_\text{out}=\emptyset$ where $Q_\text{in}$ (respectively, $Q_\text{out}$) is a set in which we place squares that are fully on the inside of $R$ (resp., fully on the outside of $R$).
Let $n$ be the smallest power of $2$ that is larger or equal to $2p$, consider an $\sqrt{2} n\times \sqrt{2}n$ square that fully contains $R$, and let $S_0$ be a set that contains only this square.

We repeat the following process for $t=\log_2n$ iterations:
\begin{itemize}[leftmargin=0.135in]
\item At the beginning of iteration $i$ we consider the set $S_{i-1}$, which contains interior-disjoint squares of size $\sqrt{2}n/2^{i-1} \times \sqrt{2}n/2^{i-1}$ that were obtained in the preceding iteration. We partition each of these squares into four interior-disjoint squares of size $\sqrt{2}n/2^{i} \times \sqrt{2}n/2^{i}$. Denote the set of these $4|S_{i-1}|$ squares as $T_{i}$. For each square of $T_{i}$, we record the four squares of $Q_\text{in}\cup Q_\text{out} \cup T_i$ that share a boundary with it (some of these might not be squares but the area outside of the square of $S_0$). This can be done in constant time by using the information that was stored in the previous iteration for the squares of $T_{i-1}$.
\item We travel across $\partial R$. Every time that we enter a square $s\in T_i$, we mark the point from which we entered $s$, mark what side of this boundary point of $s$ is the interior of $R$, and insert $s$ into $S_i$.
\item For every square $s\in T_i\setminus S_i$, we check whether $s$ is in the interior or in the exterior of $R$ (see below for the full details of this process). If $s$ is in the interior, we add $s$ to $Q_\text{in}$. Otherwise, we add it to $Q_\text{out}$.
\end{itemize}

After $\log_2n$ iterations, we have a set $Q_\text{in}$ of interior-disjoint squares that are fully contained in $R$ and a set $S_{\log_2n}$ of $\sqrt{2}\times\sqrt{2}$ squares whose interior is intersected by $\partial R$. We split every square $s\in S_{\log_2n}$ into four right-angled triangles and insert into $S$ the triangles that are in the interior of $R$ (out of the four triangles, between one and three are in the interior). After also inserting $Q_\text{in}$ into $S$, the set $S$ is a subdivision of $R$ into interior-disjoint squares and right-angled triangles.

We now explain how, at the end of the $i$-th iteration, we go over each square $s\in T_i\setminus S_i$ and check whether $s$ is in the interior or in the exterior of $R$.
We go over the squares of $T_i\setminus S_i$ in an arbitrary order. When considering a square $s \in T_i\setminus S_i$, we  already know which squares of $Q_\text{in} \cup Q_\text{out} \cup T_i$ share a boundary with $s$. Notice that  there exists a unique square in each side, and that these four squares may be of different sizes.
\begin{itemize}[leftmargin=0.135in]
\item If one of the four surrounding squares is in $Q_\text{in}$, we add $s$ to $Q_\text{in}$.
\item Otherwise, if one of the four surrounding squares is in $Q_\text{out}$, we add $s$ to $Q_\text{out}$.
\item Otherwise, if one of the four neighboring squares $s'$ is in $S_i$, we travel along the boundary of $s'$ until we get to an intersection with the border of $R$ (we marked these intersection points when inserting $s'$ to $S_i$). For each such intersection point we previously marked which side is the interior of $R$, and we can use this information to determine whether $s$ is on the outside or on the inside of $R$ (and then place $s$ accordingly in $Q_\text{in}$ or in $Q_\text{out}$).
\item We remain with the case where the four neighboring squares are currently in $T_i\setminus ( S_i \cup Q_\text{in} \cup Q_\text{out})$. In this case, we arbitrarily choose one of these four neighbours $s'$ and add $s$ to the ``waiting list'' of $s'$ (see below for the purpose of this list).
\end{itemize}

If several squares of $T_i\setminus S_i$ form a connected component, then either all of these squares are in the interior of $R$ or all of these square are in the exterior of $R$.
Thus, each time that we decide whether a square $s\in T_i\setminus S_i$ goes into $Q_\text{in}$ or into $Q_\text{out}$, we inspect the waiting list of $s$ and place the squares that are in it in the same $Q$ (we then have to check the waiting lists of each of these squares, and so on).

\subsection*{The running time of the algorithm.} Notice that $2p\le n<4p$, so for any asymptotic bound that we derive with respect to $n$, we may replace $n$ with $p$. To show that the running time of the algorithm is $O(p\log p)$, we require following lemma. Recall that $S_i$ is the set of interior-disjoint squares of size $\sqrt{2}n/2^{i} \times \sqrt{2}n/2^{i}$ at step $i$  whose interior is intersected by $\partial R$.

\begin{lem} \label{le:SiCard}
$|S_i| \le 9\cdot 2^{i-1}$.
\end{lem}
\begin{proof}
Partition the square of $S_0$ into $4^{i}$ interior-disjoint squares of size $\sqrt{2}n/2^{i} \times \sqrt{2}n/2^{i}$, and denote the set of these squares as $S'_i$. Notice that $S_i \subset S'_i$. Specifically, $S_i$ consists of the squares of $S'_i$ that are intersected by $\partial R$. We traverse $\partial R$ starting from an arbitrary point $v_1$. During this process we will mark fewer than $9\cdot 2^{i-1}$ squares, so that the marked squares fully contain the boundary of $R$. This would immediately imply $|S_i| < 9\cdot 2^{i-1}$.

We first mark a square of $S'_i$ that contains $v_1$ (there are at most four such squares), and the eight squares that surround it (i.e., share a vertex with it). We then continue to travel across the boundary of $R$ until we get to a point $v_2$ that is not contained in any marked square. We mark a square of $S_i'$ that contains $v_2$ and the eight squares surrounding it (some of these squares are already marked, and remain so). We then continue to travel until we reach a point $v_3$ that is not in any marked square. We continue in the same manner until we return to $v_1$.

Notice that each time that we get to a point $v_i$ that is in no marked square, we mark fewer than nine unmarked squares. After marking these squares, we travel at least $2n/2^{i}$ steps along $\partial R$ before we reach $v_{i+1}$. Since $\partial R$ is of length $p\le n$, the total number of marked squares is smaller than $9\cdot 2^{i-1}$
\end{proof}

The algorithm consists of $t=\log_2n$ iterations. Let us show that each iteration has a running time of $O(n)$, which would complete the proof of the theorem.  Consider the running time of the $i$-th iteration. By lemma \ref{le:SiCard}, we start this iteration with a set $S_{i-1}$ of $O(2^{i})$ squares, and partition it into a set $T_i$ of $O(2^{i})$ squares. For each new square we also record the four squares of $Q_\text{in}\cup Q_\text{out} \cup T_i$ that are its direct neighbors. Since handling each square of $T_i$ requires constant time, this step takes $O(2^i) =O(n)$ time.

We then travel the boundary of $R$, and  every time that we cross to a different square of $T_i$ we perform a constant number of operations (marking the entry point and possibly inserting the square into $S_i$). By considering the origin to be the bottom left corner of the square of $S_0$, we can easily decide when we enter a new square of $T_i$. This occurs exactly when the $x$ or $y$ coordinate of our current position becomes $0 \mod (n/2^i)$. Thus, the entire traversal of the boundary of $R$ takes $O(n)$ time.

The last part of the $i$-th iteration involves going over each square $s\in T_i\setminus S_i$ and checking whether $s$ is fully in the interior or in the exterior of $R$. This check is based on the the four squares that surround $s$. The only case that takes more than a constant time occurs when none of these squares is in $Q_\text{in}$ and $Q_\text{out}$, while at least one is in $S_i$. In this case we travel along the boundary of such a neighboring square. The perimeter of such a square is $4n/2^i$, so each instance of this case takes $O(n/2^i)$ times. By lemma \ref{le:SiCard}, $|S_i|=O(2^i)$ and we consider each square of $S_i$ at most four times (at most once for each of its four direct neighbors). Thus, the combined time of all of these checks is $O(n)$.

The only issue that we did not consider so far is the time required to handle the waiting lists. Since each square of $T_i$ is in at most one such list, and since $|T_i|=O(2^{i})$, the total time for handling the waiting lists is $O(2^i) = O(n)$.  Finally, It is easy to see that the last step of the algorithm, of cutting the squares of $S_{t}$ into triangles, requires $O(n)$ time.  This completes the proof.

\subsection*{Bounding $|S|$.}
By Lemma \ref{le:SiCard}, in the $i$-th iteration the algorithm adds fewer than $|S_i| < 9\cdot 2^{i-1}$ squares to $Q_\text{in}$. Summing this quantity over the $\log_2(n)$ iterations of the algorithm yields $|S|=O(n) = O(p)$, as asserted. Some of these squares may be split into two triangles, but this does not affect the asymptotic size of $S$.
\end{proof}

%

\begin{proof}[Proof of Theorem~\ref{t:main}]
We begin by running the algorithm of Theorem~\ref{th:Subdivision}, to obtain a subset $S\subset R$ of $O(p)$ interior-disjoint squares and triangles that cover $R$.
We build a graph $G=(V,E)$, where $V$ consists of the points of $\partial R$ and the vertices of the squares and triangles of $S$. An edge $(x,y)\in V^2$ is in $E$ if and only if $x\approx_S y$.
Notice that $|V|=O(p)$.

\subsection*{Bounded degrees.} We now prove that every vertex of $G$ is of degree at most eight. That is, that any point $x \in S$ satisfies $x\approx_S y$ for at most eight points $y\in S\setminus\{x\}$.

\begin{figure}[h]
\centering
\includegraphics[scale=0.5]{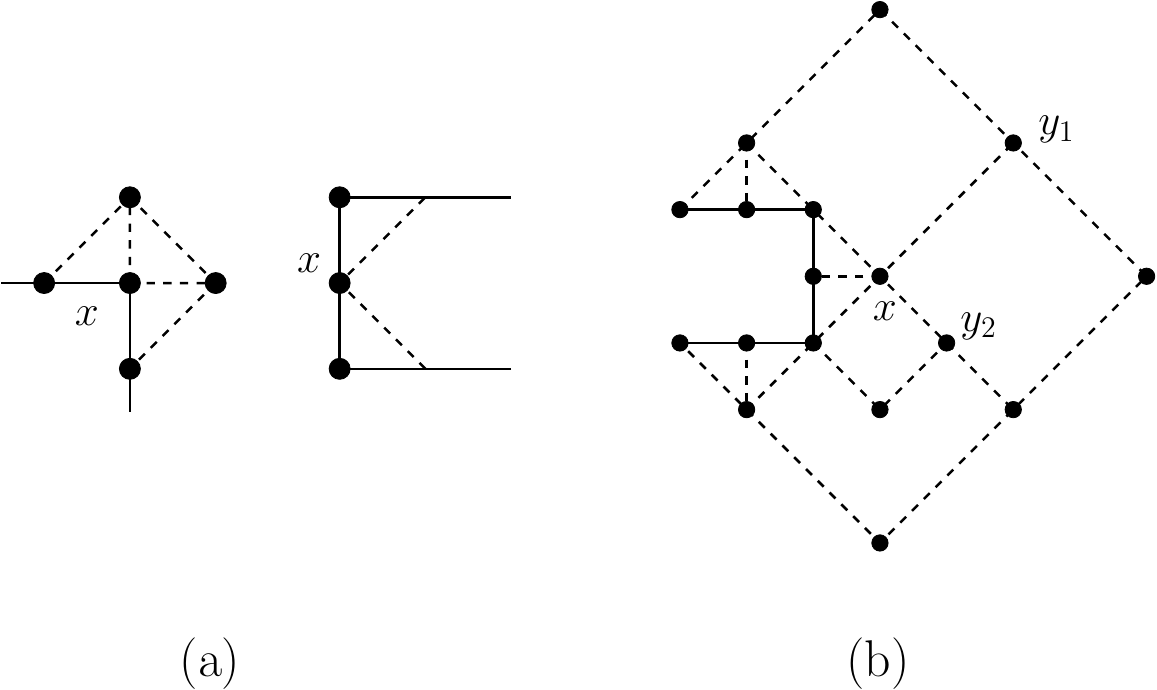}
\caption{(a) When $x\in \partial R$, it is blocked by triangles. (b) The point $x$ is the corner of three squares and one pair of triangles of $S$.}
\label{fi:DegreeTriangle}
\end{figure}

We first consider the case where $x\in \partial R$, which forces $x$ to be adjacent to triangles of $S$ (and to at most one square); for example, see Figure \ref{fi:DegreeTriangle}(a). We say that $x$ forms a \emph{valid pair} with $y\in V$ if $x\approx_S y$.
In the current case, $x$ forms a valid pair with vertices that are are in a common triangle with it. Moreover, $x$ forms a valid pair with a point $y\in V$ that does not share a triangle with it if and only if the straight-line segment between $x$ and~$y$ is fully in the interior of~$R$, does not contain any other points of $V$, and has a slope~{$\pm 1$}. Thus, in this case $x$ participates in at most seven valid pairs.

Next, consider the case where $x \notin \partial R$ and is at the corner of four squares and/or pairs of triangles of $S$; for example, see Figure \ref{fi:DegreeTriangle}(b). As before, $x$ creates a valid pair with each vertex of $V$ that shares a triangle with $x$. The maximum degree of eight is obtained when $x$ is surrounded by four pairs of triangles. If $x$ is a vertex of a square $s\in S$, denote the two edges of~$s$ that are adjacent to $x$ as $e_1$ and~$e_2$. Notice that $x$ creates a valid pair with the point of $V$ that is closest to it along $e_1$ and with the point of $V$ that is closest to it along~$v_2$. For example, in the case of the square to the right of $x$ in Figure \ref{fi:DegreeTriangle}(b), these vertices are $y_1$ and~$y_2$.

Beyond the valid pairs that are described in the previous paragraph, $x$ cannot create a valid pair with any other point of $V$. For example, in Figure \ref{fi:DegreeTriangle}(b) $x$ cannot form a valid pair with any additional vertex $z$ that is to its right, since there must be a geodesic path between $x$ and~$z$ that contains either $y_1$ or~$y_2$.

\begin{figure}[h]
\centering
\includegraphics[scale=0.45]{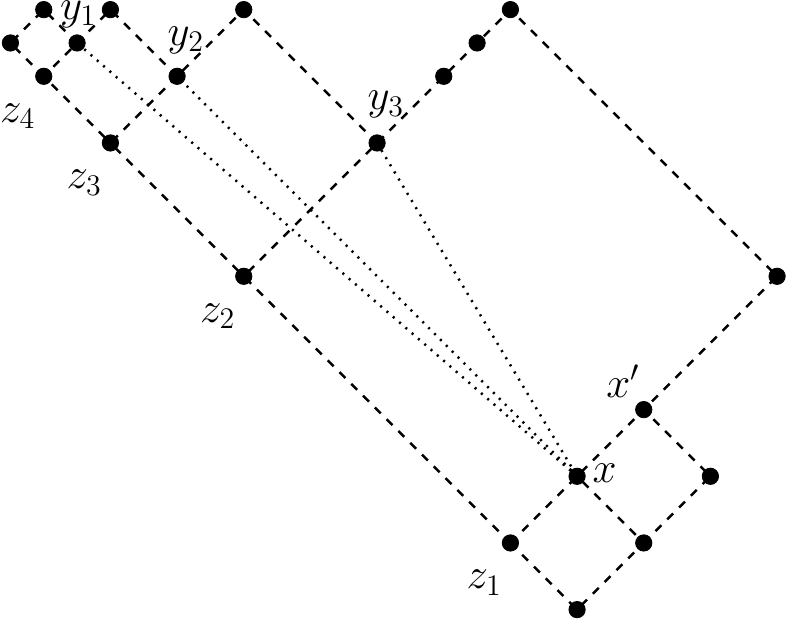}
\caption{The case where $x$ is on the boundary of a square $s\in S$ without being a vertex of $s$.}
\label{fi:DegreeSquare2}
\end{figure}

Finally, it is possible that $x \notin \partial R$ and $x$ is on the boundary of a square $s\in S$ without being a vertex of $s$; for example, see Figure \ref{fi:DegreeSquare2}. In the figure, $x$ does not form a valid pair with $z_2$, $z_3$, and $z_4$, due to geodesic paths that contain $z_1$. While $x\approx_S y_2$, we have $x\not\approx_S y_1$ due to a geodesic path that contains $y_2$. Similarly, $x\not\approx_S y_3$ due to $x'$. By the way in which we perform our subdivision of $R$ into squares, it is impossible to have a subdivision with $x$ but without $x'$ (that is, the square below $x$ cannot exist without the square to the right of $x$. The latter square may be further subdivided). Similarly, $x'$ forms a valid pair with $y_3$ but not with $y_2$, $y_1$, and the vertices to the right of $y_3$.

The examples in the previous paragraph illustrate a general principle: When $x$ is on the boundary of a square $s\in S$ without being a vertex of $s$, out of the points $y\in V$ for which the segment $xy$ intersects the interior of $s$, at most one point creates a valid pair with $x$. Specifically, such a point $y$ creates a valid pair with $x$ if and only if the segment $xy$ has slope $\pm 1$, does not contain any other point of $S$, and is fully in $R$. 
For any other such point $z\in V$, there must be a geodesic path between $x$ and $z$ that passes either through $y$ or through one of the two neighbors of $x$ along the boundary of $s$.
Thus, in this case $x$ is of degree at most six (and this degree is obtained when $x$ is adjacent to two pairs of right-angled triangles).

\subsection*{Computing $E$.}
By the above degree restriction, we have $|E|=O(p)$. To build $E$, we go over each vertex of $x\in V$ and look for the other points of $V$ that form a valid pair with $x$. By considering the above cases, we notice that if $x$ forms a valid pair with $y\in V$, then either $x$ and $y$ are on a common triangle of $S$ or the segment $xy$ has a slope of $\pm 1$ and no other points of $V$ on it.
To handle the former case, we simply go over every triangle in $S$ and add its three edges to~$E$.

For every line $\ell$ of slope 1 that contains points of $V$, we keep an array of the points of $V$ that are on $\ell$, sorted by their $x$-coordinate. There are $O(p)$ lines with a total of $O(p)$ points on them. Thus, the arrays can be built in $O(p\log p)$ time. We then go over every array and add an edge between every two adjacent points on it, with the following exception. If we get to a point $x \in \partial R$ on $\ell$, we check whether $\ell$ leaves $R$ in $x$ and if so do not add the edge that intersects the outside of~$R$.
We repeat the same process for lines with slope -1, which completes the construction of~$E$. Notice that this construction takes $O(p\log p)$ time.

By Theorem \ref{tilability}, $R$ is tileable if and only if there exists a height function $g: S \to \Z$ that satisfies \eqref{eq:condition} for every $x,y\in S$ with $(x,y)\in E$.
We now describe an algorithm for finding such a function $g$ (or stating that such a function does not exist) in $O(p \log p)$ time.
Specifically, out of the set of functions that satisfy the above condition we find the maximum function $h_{\max}$ defined in~\eqref{eq:hmax}.

\subsection*{Computing $g$.}
We begin the algorithm by initializing several variables. Let $h:\partial R \to \Z$ be a valid height function. The beginning of Section \ref{s:tileability} explains how to find such a function in $O(p)$ time (if such a function does not exist, we stop the algorithm and announce that $R$ is not tileable). Let $A$ be an array with a cell for every point $x\in S$, such that eventually we would have $A[x]=g(x)$ (that is, $A$ would describe $g$). We initially place ``N'' in each cell of $A$, to state that $g$ is currently undefined for the point that corresponds to the cell. Then, for every $x\in \partial R$ we set $A[x]=h(x)$.

Let $H$ be a heap with an element for every point of $S$ (a standard binary heap would suffice). For each $x\in S$ with $A[x]=\text{N}$, the key $H[x]$ of $x$ is the maximum integer value for $g(x)$ that does not violate \eqref{eq:condition} with points that already have a value in $A$. At first, we insert every point of $S$ to $H$ with key $H[x]=\infty$. Every time that we update a cell $A[x]$ (including during the above initialization of the points of $\partial R$ in $A$), we remove $x$ from $H$ and update the keys of each $y\in S$ that is adjacent to $x$ in $G$ (that is, for which $x \approx_S y$); specifically, for every such $y$ we set $H[y]= \min\ts\{ H[y], A[x]+\alpha (x,y) \}$. Notice that setting a value in a cell of $A$ leads to updating at most eight elements of $H$.

The main part of the algorithm consists of repeating the following process until the heap $H$ is empty and no cell of $A$ contains the undefined value N. Let $x$ be the point with the smallest key $k$ in $H$. We set $A[x]=k$, remove $x$ from $H$, and update the keys of points that are adjacent to $x$ in $G$ as described above. We then go over the vertices that are adjacent to $x$ in $G$ and already have values in $A$. For each such vertex $y$, we check whether $x$ and $y$ satisfy \eqref{eq:condition}. If not, then we stop the algorithm and announce that $R$ is not tileable.

If the above process ended since the heap $H$ is empty and no cell of $A$ contains the undefined value N, then we obtain a function $g$ that satisfies \eqref{eq:condition}. In this case, we announce that $R$ is tileable.

\subsection*{Correctness.}

To prove that the algorithm is correct, it suffices to prove that the function $g$ that the algorithm computes is indeed the maximum height function $h_{\max}$ from \eqref{eq:hmax} (although defined only on the points of $S$). By Theorem \ref{tilability}, $R$ is tileable if and only if such a function exists.

We first claim that for each $y\in S\setminus \partial R$, there exists a point $x\in S$ such that $x\approx_S y$ and $g(y)-g(x) =\alpha (x,y)$. Indeed, if no such $x$ exists then the algorithm would have assigned a larger value to $g(y)$ after removing $y$ from $H$. Since $\alpha(x,y)>0$ when $x\neq y$, we obtain that $g$ has no local minimum (with respect to the edges of $G$) outside of $\partial R$. Specifically, a straightforward induction on the (edge) distance of $x$ from $\partial R$ shows that for every $y$ in the interior of $R$ there is a path $(x_{1},\ldots,x_{n})$ in $G$ such that $y=x_{1}$, $x_{n} \in \partial R$, and $g(x_{i+1})-g(x_{i}) = \alpha (x_{i+1},x_{i})$ for every $1\le  i < n$.

We show that $g=h_{\max}$ on $S$ by induction on the number of values that the algorithm already set in $A$. For the induction basis, the claim follows by definition for every point of $\partial R$. For the induction step, assume that in the $i$-th iteration of the algorithm the point $y\in S \setminus \partial R$ is chosen, since it has a minimum key in $H$.
By the induction hypothesis, this key is
\[ H[y] \. = \. \min_{} \. \bigl(A[x] + \alpha(x,y)\bigr)
\. = \. \min_{} \, \bigl(h_{\max}(x)+\alpha(x,y)\bigr), \]
where the minima are over all $x$ s.t.~$(x,y)\in E$ and $A[x] \neq \text{N}$.

From above, there exists a path $(x_1,\ldots, x_n)$ in $G$ such that $y=x_{1}$, $x_{n} \in \partial R$, and $g(x_{i+1})-g(x_{i}) = \alpha (x_{i+1},x_{i})$ for every $1\le  i < n$. That is, $H[y] = h(x_n)+\alpha(x_n,y)$. By inspecting the definition of $h_{\max}$ in \eqref{eq:hmax}, we notice that to have $h_{\max}(y) = h(x_n)+\alpha(x_n,y)$ it remains to prove that there is no point $z \in \partial R$ such that $h(x_n)+\alpha(x_n,y) > h(z)+\alpha(z,y)$. A priori, this can only happen if the value $h(z)+\alpha(z,y)$ was not yet discovered by the algorithm since some geodesic path between $y$ and~$z$ contains a point $z'\in S$ that is still in the heap~$H$.

To see why the problematic scenario is impossible, notice that in the $i$-th iteration the key of every point $x$ that is still in $H$ is at least $h(x_n)+\alpha(x_n,y)$; otherwise we would have removed $x$ from $H$ before removing $y$. While some of these keys may be decreased in following steps of the algorithm, no key will be decreased to a value that is smaller than $h(x_n)+\alpha(x_n,y)+1$. Thus, when removing $z'$ from $H$, we have $H[z']\ge h(x_n)+\alpha(x_n,y)$. This in turn implies $h(z)+\alpha(z,y) \ge H(z')+\alpha(z',y) > h(x_n)+\alpha(x_n,y)$. That is, the problematic scenario cannot occur, and the correctness proof is complete.

\subsection*{Running time of the algorithm}

As already mentioned, obtaining $S$ requires $O(p\log p)$ time, building $G$ requires $O(p\log p)$ time, and $h$ is computed in $O(p)$ time.
Computing $g$ (that is, computing $A$) requires $O(p)$ steps. Every step involves a constant number of operations. Most of these operations require a constant time, except for removing the minimum element from the heap and updating the keys of at most eight other elements. Each such operation takes $O(\log n)$ time, so computing $g$ requires $O(p\log p)$ time.
\end{proof}

\vspace{2mm}


\begin{proof}[Proof of Theorem \ref{t:main-inter}. ]
We begin by describing the preprocessing step. In this step, we first  run the algorithm of Theorem \ref{t:main}, to obtain a subset $S\subset R$ of vertices of $O(p)$ interior-disjoint squares and triangles that cover $R$. We also obtain  the values of the maximum height function for the points of $S$. As stated in Theorem \ref{t:main}, this can be done in $O(p\log p)$ time. We then preprocess the subdivision of $R$ for point location queries (see e.g.~\cite{EGS}). Specifically, after a preprocessing time of $O(p)$, for any point $x\in R$ we can find the square or triangle that contains~$x$ in $O(\log p)$ time.

We now move to describe the query step, where we are given a query point $x\in R$. First, we consider the case where $x$ is not on the boundary of any square in the subdivision of $R$. By using the point location algorithm, we find the square that contains $x$ in $O(\log p)$ time. We then partition this square into four subsquares of equal size, and add the vertices of these subsquares into $S$. There are at most five new vertices, and we compute the height value of each in $O(\log p)$ time, as in the proof of Theorem \ref{t:main} (using the $O(p)$ arrays that we built in that proof). Out of the four subsquares, we find the one that contains $x$ (again, assuming that $x$ is not on the boundary of any of them) and subdivide it into four squares as before. We repeat this process until $x$ is surrounded by eight vertices of $S$, as in Figure \ref{fi:addedSquares}. We then have the height values of $x$ and of the eight vertices that surround it. By the conditions of Lemma \ref{le:CorrespondinghFunc} this implies the behavior of the maximum tiling of $R$ around $x$.

\begin{figure}[h]
\centering
\includegraphics[scale=0.45]{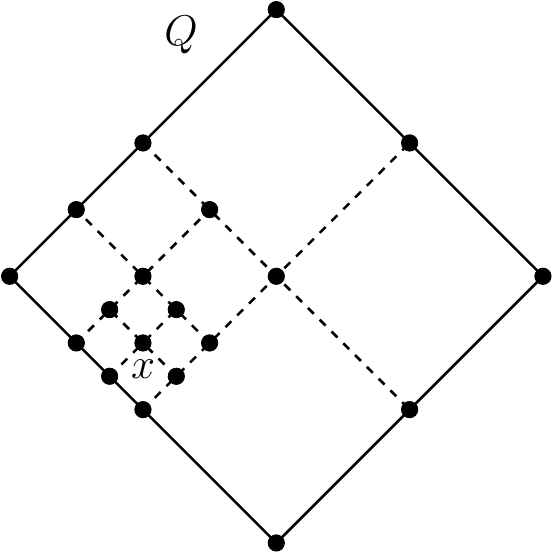}
\caption{The edges of the square $Q$ are solid, while the edges that are added in the subdivision steps of the query are dashed.}
\label{fi:addedSquares}
\end{figure}

If at some point during the above process $x$ is on the boundary of a square, we continue the subdivision process with two other points --- the one immediately above $x$ and the one immediately below $x$. Notice that such a split from $x$ into two other points can occur at most once, and that at the end of the process we still obtain the height values of $x$ and the eight vertices that surround it.

In summary, the query algorithm stops after $O(\log p)$ steps, each taking $O(\log p)$ time. Thus, handling a query requires $O(\log^2 p)$ time.
\end{proof}

\begin{rem}{\rm
We believe that with slight modifications the running time of the query can be improved to $O(\log p)$, but chose not to pursue this direction at this point.}
\end{rem}

\bigskip

\section{Lozenge tilings}\label{s:other}

\nin
The results of this paper can be extended to other lattices, and specifically to
\emph{lozenge tilings} in the triangular grid, which are dual to perfect matchings
in a \emph{hexagonal grid}.  In this section we present a brief outline for how to extend our result to the case of lozenge tilings.
Once again we follow Thurston, who defined the corresponding
height function in the original paper~\cite{Thu} (see also~\cite{Cha,R1}).


\begin{figure}[h]
\centering
\includegraphics[scale=0.21]{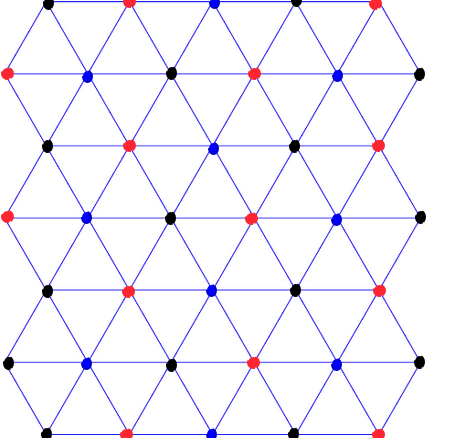}
\caption{Coloring the triangular grid with three colors.}
\label{fi:coloring}
\end{figure}

Let $\tT$ be the triangular grid with a fixed coloring of the vertices in black, red, and blue, such that every edge is adjacent to two vertices of different colors (see Figure \ref{fi:coloring}).
Let $v_1$ be a vector of length 1 in the positive direction of the $x$-axis, let $v_2$ be $v_1$ rotated counterclockwise by $120^\circ$, and let $v_3$ be $v_1$ rotated clockwise by $120^\circ$. Every point of $\tT$ can be written as $av_1+bv_2+cv_3$ (where $a,b,c$ are non-negative integers) in infinitely many ways. For example, the origin can be written as $(a,a,a)$ for every non-negative $a$. However, after adding the condition $\min\{a,b,c\}=0$, for every point $p$ of $\tT$ there is a unique way of writing $p$ as $av_1+bv_2+cv_3$. Using these unique values of $a,b,c$, we say that the coordinates of $p$ are $(a,b,c)$.

Let $R$ be a simply connected region in~$\tT$.
Similarly to the characterization of height functions for domino tilings in Lemma \ref{le:CorrespondinghFunc}, a function $h: R \to \zz$ is the height function of a lozenge tiling if and only if:
\begin{enumerate}[label={(\roman*)}]
\item For every two vertices $x,y \in R$, we have $h(x)-h(y) = 0 \mod 3$ if and only if $x$ and $y$ have the same color.
\item For every edge $(x ,y) \in \partial R$ with respective colors (black,red), (red,blue), or (blue,black), we have $h(x)-h(y)= 1$.
\item For every edge $(x,y)\in R$, we have $|h(x)-h(y)|\leq 2$.
\end{enumerate}

Also as in the case of domino tilings, for any $x\in R$ there exists a maximal height function of the plane $\alpha(x,\cdot)$ with $\alpha(x,x)=0$. Specifically, $\alpha (x,y)$ is defined as the sum of the three coordinates of  $y$ when considering $x$ as the origin. In the triangular grid, we say that a sequence of points $(x_{1},\ldots, x_{n})$ is a geodesic path if
\begin{itemize}
\item  For every $i < n$, we have $\| x_{i+1} - x_1 \|_{1} = \| x_{i} - x_1 \|_{1}$+1.
\item  For every $i < n$, the points  $x_{i}$ and $x_{i+1}$ are corners of a common $1\times 1$ triangle in $\tT$.
 \end{itemize}

Unlike the case of $\Z^2$, in the triangular grid geodesic paths do not necessarily minimize the number of edges. Moreover, when travelling along a geodesic path, we only move in the directions $v_1,v_2,v_3$ (since moving in one of the other three directions will result in a step of distance 2). Specifically, every geodesic path uses at most two of these three directions.
As before, for $x,y\in  \tT$ we denote by $G(x,y)$ the union of the geodesic paths between $x$ and~$y$.  It is not difficult to verify that $G(x,y)$ is always a parallelogram (possibly of width zero). Let $S$ be a set that contains $\partial R$ and any number of points from the interior of~$R$.
We write $x\approx_S y$ when $x,y \in S$ and $G(x,y)\setminus\{x,y\}$ is disjoint from~$S$. The following theorem is the lozenge tiling analogue of Theorem \ref{tilability}.

\begin{thm}
\label{tilability2}
Let $R$ be a simply connected region in $\tT$ that contains the origin and let $S \subset R $ be a set that contains $\partial R$.  Then $R$ is tileable if and only if there exists $g:S \to \Z$ such that $g=h$ on $\partial R$ and for every pair $x,y \in  \partial S$ with $x\approx_{S} y$ we have
\begin{equation} \label{eq:condition2}
-\alpha (y,x) \leq g(y) - g(x) \leq \alpha (x,y).
\end{equation}
\end{thm}

The proof of Theorem \ref{tilability2} follows verbatim the proof of Theorem~\ref{tilability}.  We omit the details.
Similarly, we can use a variant of Theorem \ref{th:Subdivision} to find a subdivision of $R$ into $O(p)$ interior-disjoint triangles of various sizes. The main difference is that here we use equilateral triangles with side length $2^{i}$ instead of squares. Indeed, every such triangle can be subdivided into four interior-disjoint triangles of side length $2^{i-1}$. We set $S$ to consist of $\partial R$ together with the vertices of the triangles of the subdivision. For any $x \in S$, it is not difficult to show that at most six points $y\in S$ for which $x\approx_S y$ (the actual bound seems to be smaller than six, but this does not matter for our purpose).
Finally, by revising the algorithm of Theorem \ref{t:main}, we obtain the following theorem.

\begin{thm} \label{t:main2}
Let $\Ga$ be a simply connected region in the triangular grid of the plane, and
let $p = |\pt \Ga|$ be the perimeter of $\Ga$.
Then there exists an algorithm that decides tileability of~$\Ga$
in time $O(p \ts\log\mts p)$.
\end{thm}

\bigskip

\section{Final remarks and open problems} \label{s:fin}

\subsection{}\label{ss:fin-match}
For general bipartite planar graphs, recent developments improve
the Hopcroft--Karp bound~\cite{HK} to nearly linear time.  First, the
existence of a perfect matching is equivalent to the \emph{circulation
problem}, where all white vertices have supply~1 and all black vertices have
demand~1.  It is known that the circulation problem in planar graphs can be
solved within the same time bound as the \emph{shortest path problem}
with negative weights on a related planar graph~\cite{MN}.  The latter
can be solved in time $O(n\log^2 n/\log\log n)$, see~\cite{MW}. This
almost matches Thurston's original $O(n\log n)$ bound in
Theorem~\ref{t:thurston}.

\subsection{}\label{ss:fin-oracle}
Our oracle model for the perfect matching is similar to other models
of sparse graph presentations, which are popular in the study of
graph properties of massive graphs (see e.g.~\cite{Gol1}).  The idea
is to give a sublinear size presentation of a perfect matching,
amenable to running further sublinear time algorithms;  see e.g.~\cite{RS}
for a primer on the subject.

\subsection{}\label{ss:fin-tile}
Thurston originally defined and studied height functions for the domino
tilings as in this paper and for the lozenge tilings in a triangular lattice.
Since then, many generalizations and variations have been discovered.
These include other rectangles in the plane~\cite{KK,Korn,R2}, other tiles
in the triangular lattice~\cite{R1}, rhombus tilings in higher dimension~\cite{LMN},
perfect matchings of more general graphs in the plane and other
surfaces~\cite{Cha,Ito,STCR}, and even infinite domino tilings~\cite{BFR}.
We refer to~\cite{Pak} for a (somewhat dated) survey of various tileability
applications of height functions and tiling groups.

On the complexity side, there are a number of $\NP$-completeness
results for the decision and counting problems for general regions with small tiles,
see e.g.~\cite{BNRR,MR}, and more recently for simply connected regions~\cite{PY1}.
In case of domino tilings, there are also $\SP$-completeness results for
$3$-dimensional regions~\cite{PY2,V2}.

The notion of height functions for domino and lozenge tilings has also made a
remarkable impact in Probability and MCMC studies (see e.g.~\cite{LRS,Ken}).

\subsection{}\label{ss:fin-cs}
The complexity of Thurston's algorithm has been investigated to a remarkable
degree in the Computational Geometry literature.  These include generalizations
to regions with holes~\cite{Thi}, parallel computing~\cite{Fou}, and more
general graphs~\cite{Cha}.

The idea behind our tileability criterion, stated in Lemma~\ref{l:necessary},
was first given in the third author's thesis~\cite{Tas}, in the context of
tromino tilings.  The criterion is especially surprising given the
fundamentally non-local property of the domino tileability, as elucidated
by the \emph{augmentablity problem} (see~\cite[$\S$11.3]{Korn}).

\subsection{}\label{ss:fin-other}
We believe that our approach can be further extended to a variety of
tiling problems which admit height functions, such as tilings with bars
(see~\cite{BNRR,KK,Tas}).  In a different direction, the heart of the
proof is the idea of \emph{scaling} represented by the squares which are
used heavily in Section~\ref{s:algorithm}.  It would be nice to see
this idea can be further developed.  Finally, if the boundary $\pt \Ga$ is given
by some kind of periodic conditions (cf.~\cite{Ken-boundary}),
one can perhaps further speed up the domino tileability testing.
Unfortunately, at the moment, we do not know how to formalize
this problem.

\vskip.3cm

\noindent
{\bf Acknowledgements.}  We are very grateful to Scott Garrabrant
and Yahav Nussbaum  for interesting discussions and helpful remarks.
The first author was partially supported by the~NSF.

\vskip.6cm


\end{document}